\documentclass[a4paper,12pt,reqno]{amsart}
\author{Julia Brandes}
\title[Linear spaces on hypersurfaces]{Linear spaces on hypersurfaces \\ over number fields}
\address{Mathematical Sciences, Chalmers Institute of Technology and University of Gothenburg, 412 96 G{\"o}teborg, Sweden}
\email{brjulia@chalmers.se}

\usepackage[english]{babel}
\usepackage[T1]{fontenc}
\usepackage[latin1]{inputenc}
\usepackage{latexsym}
\usepackage{epsfig}
\usepackage{bm}
\usepackage{amssymb}
\usepackage{amsmath}
\usepackage{verbatim}
\usepackage{amsthm}
\usepackage{enumerate}
\usepackage{mathrsfs}
\usepackage[hmargin=3cm,vmargin=3cm]{geometry}

\makeatletter
\def\Ddots{\mathinner{\mkern1mu\raise\p@
\vbox{\kern7\p@\hbox{.}}\mkern2mu
\raise4\p@\hbox{.}\mkern2mu\raise7\p@\hbox{.}\mkern1mu}}
\makeatother

\def\A{\mathbb A}
\def\C{\mathbb C}
\def\K{\mathbb K}
\def\PP{\mathbb P}
\def\Q{\mathbb Q}
\def\R{\mathbb R}
\def\T{\mathbb T}
\def\V{\mathbb V}
\def\Z{\mathbb Z}
\def\OK{\mathcal O_{\K}}

\def\B#1{\mathbf{#1}}
\def\ba{\bm{\alpha}}
\def\bb{\bm{\beta}}
\def\bg{\bm{\gamma}}

\def\F#1{\mathfrak{#1}}
\def\cal#1{\mathcal{#1}}
\def\D{\,\mathrm{d}}

\def\mmod#1{\;(\mathrm{mod}\;{#1})}
\def\ol#1{\overline{\B{#1}}}

\def\ba{\bm \alpha}
\def\bb{\bm \beta}
\def\bg{\bm \gamma}
\def\U#1{\underline{#1}}
\def\eps{\varepsilon}

\renewcommand\le{\leqslant}
\renewcommand\ge{\geqslant}

\DeclareMathOperator{\Tr}{Tr}
\DeclareMathOperator{\Nm}{Nm}
\DeclareMathOperator{\rk}{rank}
\DeclareMathOperator{\card}{Card}

\DeclareMathOperator{\vol}{vol}
\DeclareMathOperator{\sing}{Sing}

\newtheorem{thm}{Theorem}[section]
\newtheorem{lem}[thm]{Lemma}

\theoremstyle{definition}

\theoremstyle{remark}

\numberwithin{equation}{section}

\newenvironment{pf}{\begin{proof}[Proof]}{\end{proof}}

\begin{document}

\begin{abstract}
    We establish an analytic Hasse principle for linear spaces of affine dimension $m$ on a complete intersection over an algebraic field extension $\K$ of $\Q$. The number of variables required to do this is no larger than what is known for the analogous problem over $\Q$. As an application we show that any smooth hypersurface over $\K$ whose dimension is large enough in terms of the degree is $\K$-unirational, provided that either the degree is odd or $\K$ is totally imaginary.
\end{abstract}

\maketitle

\section{Introduction}
    One of the main developments of recent years in the study of the circle method has been an increasing interest in generalising results that have been obtained over the rationals to more general fields with an arithmetic structure such as number fields or function fields, both in order to acquire a deeper understanding of how specific the results are to the integers or integer-like objects, and in order to be able to circumvent certain restrictions imposed by the integral setting.  Some of the major efforts in this direction are due to Skinner \cite{Sk:94,Sk:97} who established number field versions of the influential papers by Heath-Brown on rational points on non-singular cubic surfaces \cite{HB:83} and by Birch on forms in many variables \cite{Bir:61}. The former paper falls somewhat short of what had been known in the rational case, but in recent work Browning and Vishe \cite{BrVis:14} found an improved treatment so that now the number field case is almost as well understood as the rational case. Similarly, the recent paper of Browning and Heath-Brown generalising Birch's theorem to systems involving differing degrees \cite{BHB:14} has immediately been translated to the number field setting by Frei and Madritsch \cite{FM}, as has Dietmann's work on small solutions of quadratic forms \cite{diet:03} by Helfrich \cite{Hel}. In this memoir we aim to continue in this direction by providing a number field version of the author's recent work on linear spaces on hypersurfaces \cite{B:14, B:15FRF2}.

    Let $\K$ be an algebraic number field of degree $n$ over $\Q$ with ring of integers $\OK$. Let $\omega_1, \dots, \omega_n$ be an integral basis of $\OK$, then it is also a $\Q$-basis of $\K$.  Consider a box
    \begin{align*}
        \cal B = \{x  \in \K: x = \widehat x_{1}\omega_1 + \dots + \widehat x_{n} \omega_n, \quad \widehat x_{i} \in [-1,1]\}.
    \end{align*}
    For a given set of polynomials $F^{(1)},\dots, F^{(R)} \in \K[x_1, \dots, x_s]$ of degree $d$ we study the number $N_m(P)$ of $m$-tuples $\B x_1, \dots, \B x_m \in (\OK \cap P \cal B)^s$ satisfying the identities
    \begin{align}\label{sys}
        F^{(\rho)}(\B x_1 t_1 + \dots + \B x_m t_m) = 0 \quad (1 \le \rho \le R)
    \end{align}
    identically in $t_1, \dots, t_m$. Set $r = \binom{d-1+m}{d}$, and let
    \begin{align*}
        \sing^*(\B F) = \left\{ \B x \in \A_\K^s: \rk \left(\partial F^{(\rho)}(\B x)/\partial x_i \right)_{\rho, i} \le R-1 \right\}.
    \end{align*}
    As in comparable work, our methods are equally strong over number fields as they are over the rationals.

    \begin{thm}\label{thm:main}
      Let $F^{(1)},\dots, F^{(R)}$ be as above, and suppose that $m$ and $d\ge 2$ are integers and that
      \begin{align}\label{s-cond}
        s - \dim \sing^* \B F > 2^{d-1}(d-1)Rr(R+1).
      \end{align}
      Then there exists a non-negative constant $c$ and a parameter $\delta>0$ such that
      \begin{align}\label{eq:asymp}
        N_m(P) = c (P^{n})^{ms-rd} + O((P^{n})^{ms-rd-\delta}).
      \end{align}
    \end{thm}
    The constant $c$ has an interpretation as a product of local densities, so that Theorem~\ref{thm:main} yields an analytic Hasse principle. We also note that the case $m=1$ recovers Skinner's result \cite{Sk:94}, and for larger $m$ we save approximately one factor $r$ over what a naive application of Skinner's methods would yield, thus replicating the improvements of the author's earlier work \cite{B:14, B:15FRF2} over a naive application of Birch's theorem.
    One feature of the proof that is worth highlighting is our treatment of the singular integral. In  recent work, Frei and Madritsch \cite{FM} identified an inaccuracy in the work of Skinner \cite{Sk:97}, and proposed a corrected treatment. Unfortunately, their argument is rather involved, but we are able to give a much simplified proof of the same statement that parallels the treatment over $\Q$.

    An obvious question is under what conditions the constant $c$ is positive. This depends on the number field $\K$, but we can still state a result for a large class of fields.
    \begin{thm}\label{thm:local}
        Let $F^{(1)},\dots, F^{(R)}$ be as above, and suppose that $m$ and $d\ge 2$ are integers. Suppose further that either $d$ is odd or $\K$ is totally imaginary, and that
                $$s - \dim \sing^* \B F > 2^{d-1}(d-1)R\max\{r(R+1), d^{2^{d-1}}(R^2d^2+Rm)^{2^{d-2}}\}.$$
        Then \eqref{eq:asymp} holds with $c>0$.
    \end{thm}
    As we will see in \S \ref{s:local}, this follows from Theorem~\ref{thm:main} by applying results from the literature. Observe further that the first term in the maximum occurs for $d\le 3$ and large $m$, whereas for $d \ge 4$ the second term always dominates.  \medskip

    A consequence of Theorem~\ref{thm:local} concerns the question under what conditions a hypersurface is unirational. Two projective varieties are said to be birationally equivalent if they can be mapped onto one another by a rational map. Unfortunately, establishing birational equivalence for two given varieties is often difficult in practice, so for many applications one is satisfied with the weaker notion of unirational covers, which abandons the requirement that the rational map be an isomorphism on a Zariski-open subset and only requires a surjective cover. We call a projective variety $V$ unirational over $\K$ if there exists a dominant morphism from the pojective space $\PP_\K^{\dim V}$ onto $V$. It is straightforward to show that quadrics with a $\K$-point are always unirational over their ground field, and in a series of papers by Segre \cite{Segre:43}, Manin \cite[II.2]{Manin:Cub}, Colliot-Th\'el\`ene, Sansuc and Swinnerton-Dyer \cite[Remark 2.3.1]{CTSSD:87} and Koll\'ar \cite[Theorem~1.1]{kollar:02}, it has been shown that a smooth rational cubic hypersurface of dimension at least $2$ over any field $\K$ is unirational over $\K$ as soon as it contains a $\K$-point.

    For higher degrees the situation is more complicated. Following up on ideas by Morin \cite{Morin:42} and Predonzan \cite{Pred}, Paranjape and Srinivas \cite{ParSri:92} were able to show that a general complete intersection of sufficiently low degree is always unirational over its ground field. This has been taken one step further by Harris, Mazur and Pandharipande \cite{HMP:98}, who improved upon the almost-all-result of the former authors by showing that \emph{every} smooth hypersurface containing a sufficiently large $\K$-rational linear space is unirational over $\K$. Stating their result requires some notation. For $d \ge 2$ and $k \ge 0$ set
    \begin{align*}
        N(d,k) &= \begin{cases}
                    \displaystyle{\binom{k+1}{2}+3} & \mbox{if } d=2,\vspace*{2mm}  \\
                    \displaystyle{\binom{N(d-1, k)+d}{d-1} + N(d-1, k) + \binom{k+d}{d} + 2} & \mbox{for } d \ge 3,
                 \end{cases}
    \end{align*}
    and
    \begin{align*}
        L(d,k) &= \begin{cases}
                    0 & \mbox{if } d=2, \\
                    N(d-1, L(d-1)) & \mbox{if } d \ge 3.
                 \end{cases}
    \end{align*}
    Then Corollary 3.7 of \cite{HMP:98} shows that a hypersurface of degree $d$ over $\K$ is unirational over $\K$ if it contains a $\K$-rational plane of dimension $m \ge L(d)+1$. Hence as a consequence of Theorem~\ref{thm:local} we obtain the following.

    \begin{thm}\label{thm:unirat}
        Suppose that either $\K$ is a totally imaginary field extension or $d$ is odd, and let $F \in \K[x_1, \dots, x_s]$ be a non-singular homogeneous polynomial of degree $d \ge 4$, where
        \begin{align*}
            s > 2^{d-1}(d-1)(d^2 + L(d)+1)^{2^{d-2}} d^{2^{d-1}}.
        \end{align*}
        Then the hypersurface $F(\B x)=0$ is unirational over $\K$.
    \end{thm}

    Unfortunately, the numbers required to achieve this are very large. In fact, one can compute $L(4)=97$, $L(5) = 252694544886958321667 \approx 2.52\ldots \cdot 10^{20}$ and in general
    \begin{align*}
        L(d) \approx \underbrace{d^{d^{\Ddots^d}}}_{d \text{ times}} = d \uparrow \uparrow d.
    \end{align*}
    Accordingly, the bounds of Theorem~\ref{thm:unirat} are of size $L(d)^{2^{d-2}}$, which yields the bound $s> 265650463309824 \approx 2.65\ldots \cdot 10^{14}$ in the case $d=4$, and $s>1.62\ldots \cdot 10^{173}$ for $d=5$. One should expect that by applying ideas due to Heath-Brown \cite{HB:10} and Zahid \cite{Zahid:09} significantly sharper estimates can be obtained for these small degrees; we intend to pursue such refinements in future work.

    The author is grateful to Tim Browning for motivating this work and in particular for pointing out the application to unirationality.

\section{Notation and Setting} \label{s:not}

    Our setting over number fields demands a certain amount of notation. In our nomenclature we largely follow the works of Skinner \cite{Sk:97} and Browning and Vishe \cite{BrVis:14}. Let $n=n_1+2n_2$, where $n_1$ and $n_2$ denote the number of real resp. complex embeddings of $\K$. We denote these embeddings by $\eta_l$ with the convention that real embeddings are labelled with indices $1 \le l \le n_1$, and for $1 \le i \le n_2$ the embeddings with indices $n_1+i$ and $n_1+n_2+i$ are conjugates. Most of the time we will work over the $n$-dimensional $\R$-algebra
    \begin{align*}
        \V = \K \otimes_{\Q} \R \cong \bigoplus_{l=1}^{n_1+n_2}\K_l,
    \end{align*}
    where $\K_l$ is the completion of $\K$ with respect to $\eta_l$, so we have $\K_l = \R$ for $1 \le l \le n_1$ and $\K_l = \C$ for $n_1+1 \le l \le n_2$. Of course, $\K$ has a canonical embedding in $\V$ given by
    \begin{align*}
        \alpha \mapsto (\eta_1(\alpha), \dots, \eta_{n_1+n_2}(\alpha)),
    \end{align*}
    which allows us to identify $\K$ with its image in $\V$. By writing $\alpha^{(i)} = \eta_{i}(\alpha)$ we thus have $v = \oplus_l v^{(l)}$ for each $v \in \V$. The norm and trace on $\V$ are defined via
    \begin{align*}
       \Nm(v) &= v^{(1)} \cdot \ldots \cdot v^{(n_1)} |v^{(n_1+1)}|^2 \cdot \ldots \cdot |v^{(n_1+n_2)}|^2, \\
       \Tr(v) &= v^{(1)} + \ldots + v^{(n_1)} + 2\F R v^{(n_1+1)}+  \ldots + 2\F R v^{(n_1+n_2)}.
    \end{align*}
    Write further $\Omega(\K)$ for the set of places of $\K$, and let $\Omega_0(\K)$ and $\Omega_\infty(\K)$ denote the set of finite and infinite places, respectively.

    The image of any fractional ideal of $\OK$ takes the shape of a lattice in $\V$ as follows. If $\{\omega_1, \dots, \omega_n\}$ forms a $\Z$-basis of $\OK$, then it is  also an $\R$-basis of $\V$ and we have
    \begin{align}\label{eq:basis}
        \V = \{x= \widehat x_1 \omega_1 + \dots + \widehat x_n \omega_n: \; \widehat x_i \in \R \text{ for all }1 \le i \le n\}.
    \end{align}
    We further write
    \begin{align*}
        \OK^+ = \{x= \widehat x_1 \omega_1 + \dots + \widehat x_n \omega_n \in \OK: \widehat x_i \ge 0 \text{ for all }1 \le i \le n\}.
    \end{align*}
    In the interest of maintaining a consistent notation, we will denote elements in $\K$ by lower case letters, and denote the respective vector in $\R^n$ by hats, so that for $x \in \K$ we have
    \begin{align*}
        x = \bigoplus_{l=1}^{n_1+n_2} x^{(l)} = \widehat x_1 \omega_1 + \dots + \widehat x_n \omega_n, \qquad  \widehat{\B x} = (\widehat x_1, \dots, \widehat x_n).
    \end{align*}
    The analogue of the unit interval for the field $\K$ is given by the set
    \begin{align*}
        \T = \{x \in \V: 0 \le \widehat x_i \le 1 \quad (1 \le i \le n)\}.
    \end{align*}
    We use the volume form induced by \eqref{eq:basis}, namely $\D x = \D \widehat x_1 \cdots \D \widehat x_n$. According to this volume form, we have $\vol(\T) = 1$ as expected.
    For any element $a \in \K$ we have the denominator ideal
    \begin{align}\label{eq:denom}
        \F q(a) = \{b \in \OK: ab \in \OK \},
    \end{align}
    which is easily extended to vectors $\B a \in \K^s$ by setting $\F q(\B a) = \bigcap_{i} \F q(a_i)$. Denominator ideals are always principal, and we have
    \begin{align}\label{eq:nmineq}
        \card\{\bg \in (\T \cap \K)^R: \left| \Nm(\F q(\bg)) \right| = q\} \ll q^{R+\eps}
    \end{align}
    (see e.g. \cite[Lemma~5 (i)]{Sk:97}).

    In the embedding \eqref{eq:basis} we have the standard height function
    \begin{align*}
        |x| = \max\{|\widehat x_1|, \dots, |\widehat x_n|\},
    \end{align*}
    so that $|x|  \asymp \max_{v \in \Omega_{\infty}(\K)} |x|_v$. This norm extends in the obvious manner to vectors $\B x \in \V^s$. Furthermore, for $x \in \K$ we have $        |x^{-1}| \ll |x|^{n-1}/ \left|\Nm x\right|$.

    If $F \in \V[x_1, \dots, x_s]$ is a polynomial, we may consider the associated polynomial
    \begin{align*}
        \widehat F(\widehat{\B x}) = \Tr (F(\B x)) \in \R[\widehat x_{1,1}, \dots, \widehat x_{s,n}].
    \end{align*}
    Projecting on the basis vectors $\omega_l$, we also have the system
    \begin{align*}
        \widehat F_l(\widehat{\B x}) = \Tr(\omega_l F(\B x)) \in \R[\widehat x_{1,1}, \dots, \widehat x_{s,n}] \qquad (1 \le l \le n).
    \end{align*}
    Since we may assume the basis $\{\omega_1, \dots, \omega_n\}$ to be orthonormal, this isolates the $l$-th coefficient of $F(\B x)$ with respect to the representation \eqref{eq:basis}.\medskip

    We set up the circle method as in \cite{B:14}. The additive character over number fields is given by $e(x) = e^{2 \pi i \Tr x}$. To each polynomial $F^{(\rho)}$ we associate the unique symmetric $d$-linear form $\Phi^{(\rho)}$ satisfying $\Phi^{(\rho)}(\B x, \dots, \B x) = F(\B x)$. Write further $J  = \{1,\dots, m\}^d$ disregarding order, so that $\card J = r$. In this notation we have
    \begin{align}\label{eq:exp}
        F^{(\rho)}\left(t_1 \B x_1+ \dots + t_m \B x_m\right) =\sum_{ \B j \in J} A(\B j) t_{j_1}t_{j_2} \cdot \ldots \cdot t_{j_d} \Phi^{(\rho)}(\B x_{j_1}, \B x_{j_2},\ldots, \B x_{j_d})
    \end{align}
    for suitable combinatorial constants $A(\B j)$. Set
    \begin{align}\label{eq:Phi}
        \Phi^{(\rho)}_{\B j} (\B x_1, \dots, \B x_m) = A(\B j)\Phi^{(\rho)}(\B x_{j_1}, \B x_{j_2},\ldots, \B x_{j_d}),
    \end{align}
     and write $\ol x = (\B x_1, \dots, \B x_m) \in \V^{ms}$.
    It follows by expanding the system \eqref{sys} as in \eqref{eq:exp} that counting solutions $\B x_1, \dots, \B x_m$ to \eqref{sys} is equivalent to counting solutions $\ol x$ to the system
    \begin{equation}\label{eq:eqsys}
        \Phi_{\B j}^{(\rho)} (\ol x) = 0 \quad (1 \le \rho \le R, \B j \in J).
    \end{equation}
    We write $\ba^{(\rho)} = (\alpha_{\B j}^{(\rho)})_{\B j \in J}$ and $\U{\ba} = (\ba^{(1)}, \dots, \ba^{(R)})$. For the sake of completeness we also define $\U{\alpha}_{\B j} = (\alpha_{\B j}^{(1)}, \dots, \alpha_{\B j}^{(R)})$.
    In this notation we have
    \begin{align}
        N_m(P) = \sum_{\ol x \in P \cal B^{sm}} \int_{ \T^{Rr}} e\left(\sum_{\B j \in J} \sum_{\rho=1}^R \alpha_{\B j}^{(\rho)}\Phi_{\B j}^{(\rho)}(\ol x)\right) \D \U{\ba}.
    \end{align}
    It will be convenient to write
    \begin{align*}
        \F F(\ol x; \U{\ba}) = \sum_{\B j \in J} \sum_{\rho=1}^R \alpha_{\B j}^{(\rho)}\Phi_{\B j}^{(\rho)}(\ol x)
    \end{align*}
    and
    \begin{align*}
        T_P(\U{\ba}) = \sum_{\ol x \in P \cal B^{sm}} e(\F F(\ol x; \U{\ba})),
    \end{align*}
    so that
    \begin{align*}
        N_m(P) = \int_{ \T^{Rr}} T_P(\U{\ba}) \D \U{\ba}.
    \end{align*}
    We remark that these definitions can be brought back to $\R$. In fact, writing
    \begin{align*}
        \widehat{\F F}(\widehat{\ol x}; \widehat {\U{\ba}}) = \sum_{l=1}^n \sum_{\B j \in J} \sum_{\rho=1}^R \widehat \alpha_{\B j,l}^{(\rho)} \widehat \Phi_{\B j,l}^{(\rho)}(\widehat{\ol x}),
    \end{align*}
    where $\widehat {\U{\ba}}$ denotes the coefficient vector of $\U{\ba}$ according to \eqref{eq:basis}, we obtain
    \begin{align*}
        T_P(\U{\ba}) = \sum_{\substack{\widehat{\ol x} \in \Z^{mns}\\ |\widehat{\ol x}| \le P}} e\left(\widehat{\F F}(\widehat{\ol x}; \widehat{\U{\ba}})\right).
    \end{align*}

    Finally, we make some remarks as to the general notational conventions we shall adopt. Any statement involving the letter $\eps$ is claimed to hold for any $\eps>0$. Consequently, the exact `value' of $\eps$ will not be tracked and may change from one expression to the next. The letter $P$ is always used to denote a large integer. Since many of our estimates are measured in terms of $P^n$, we set this quantity equal to $\Pi$. Expressions like $\sum_{n=1}^x f(n)$, where $x$ may or may not be an integer, should be read as $\sum_{1 \le n \le x} f(n)$.
    We will abuse vector notation extensively. In particular, equalities and inequalities of vectors should always be interpreted componentwise. Similarly, for $\B a \in \Z^l$ we will write $(\B a, b) = \gcd(a_1, \dots, a_l, b)$.
    Finally, the Landau and Vinogradov symbols will be used in their established meanings, and the implied constants are allowed to depend on $s$, $m$, $d$ and $n$ as well as the coefficients of $F$, but never on $P$.

\section{Exponential Sums}\label{sec-Weyl}
    In this section we  study the exponential sum $T_P(\U{\ba})$ in greater detail. We define the discrete differencing operator $\Delta_{i,\B h}$ via its action
    \begin{align*}
        \Delta_{i, \B h}\F F(\ol x; \U{\ba}) = \F F(\B x_1, \dots, \B x_i + \B h, \dots, \B x_m; \U{\ba}) - \F F(\ol x; \U{\ba}).
    \end{align*}
    The following lemma is now a straightforward modification of \cite[Lemma~3.1]{B:14}.

    \begin{lem}
        Let $1 \le k \le d$. For $1 \le i \le k$ let $j_i$ be integers with $1 \le j_i \le m$. Then
        \begin{align*}
            |T_P(\U{\ba})|^{2^k} \ll P^{((2^k-1)m - k)ns} \sum_{\B h_1, \dots, \B h_k \in P \cal B^s} \sum_{\ol x} e\left(\Delta_{j_1, \B h_1} \cdots \Delta_{j_k, \B h_k} \F F(\ol x; \U{\ba}) \right),
        \end{align*}
        where the sum over $\ol x$ is over a suitable box contained in $P \cal B^{sm}$.
    \end{lem}

    Observe that in each differencing step the degree of the forms involved decreases by one, so after $d-1$ steps we arrive at a polynomial that is linear in $\ol x$.
    For the sake of simplicity we write $\cal H$ for the $(d-1)$-tuple $(\B h_1, \dots, \B h_{d-1})$. In this notation we have
    \begin{align*}
        |T_P(\U{\ba})|^{2^{d-1}}
        & \ll P^{((2^{d-1}-1)m-(d-1))ns}\sum_{\cal H} \sum_{\ol x} e\left(\Delta_{j_1, \B h_1} \cdots \Delta_{j_{d-1}, \B h_{d-1}} \F F(\ol x; \U{\ba}) \right) \\
        &\ll  P^{(2^{d-1}m-d)ns}\sum_{\cal H} \left|\sum_{\B x_{j_{d}}}  e\left(M(\B j) \sum_{\rho=1}^R \alpha_{\B j}^{(\rho)} \Phi^{(\rho)}(\B x_{j_{d}},\cal H) \right)\right|,
    \end{align*}
    where $M(\B j)$ is a suitable combinatorial constant as in \cite[Lemma~3.2]{B:14}.

    We write $\widehat{\cal H}$ for  the coefficient vector of $\cal H$ by the representation \eqref{eq:basis} and define the functions $\widehat B_{i,l}^{(\rho)} \in \Z[\widehat{\B h}_1, \dots, \widehat{\B h}_{d-1}]$ via the identity
    \begin{align*}
        \widehat{\Phi}^{(\rho)}(\widehat{\B x}, \widehat{\cal H}) = \sum_{i=1}^s \sum_{l=1}^n \widehat x_{i,l} \widehat B_{i,l}^{(\rho)}(\widehat{\cal H}).
    \end{align*}
    In this notation the above inequality can be brought back to $\R$, where it reads
    \begin{align*}
        |T_P(\U{\ba})|^{2^{d-1}}
        & \ll P^{(2^{d-1}m-d)ns} \sum_{\widehat{\cal H}}  \prod_{i=1}^s \prod_{l=1}^n \left|\sum_{\widehat x_{i,l}}  e\left(M(\B j) \sum_{\rho=1}^R  \widehat \alpha_{\B j, l}^{(\rho)} \widehat x_{i,l} \widehat B_{i,l}^{(\rho)}(\widehat{\cal H}) \right)\right| \\
        & \ll  P^{(2^{d-1}m-d)ns}\sum_{\widehat{\cal H}}  \prod_{i=1}^s \prod_{l=1}^n \min\left\{ P, \left\| M(\B j) \sum_{\rho=1}^R  \widehat \alpha_{\B j, l}^{(\rho)} \widehat B_{i,l}^{(\rho)}(\widehat{\cal H}) \right\|^{-1} \right\}  .
    \end{align*}
    Denote by $N_{\B j}(A, B)$ the number of $(d-1)$-tuples $\widehat{\B h}_1, \dots, \widehat{\B h}_{d-1} \in \Z^{ns}$ with $|\widehat{\B h}_k| \le A$ satisfying
    \begin{align*}
        \left\| M(\B j) \sum_{\rho=1}^R  \widehat \alpha_{\B j, l}^{(\rho)} \widehat B_{i,l}^{(\rho)}(\widehat{\cal H}) \right\| < B \quad (1 \le l \le n, 1 \le i \le s).
    \end{align*}
    The argument of the proof of Lemma~3.2 of \cite{B:15FRF2} shows then that
    \begin{align*}
        \sum_{\widehat{\cal H}}  \prod_{i=1}^s \prod_{l=1}^n \min\left\{ P, \left\| M(\B j) \sum_{\rho=1}^R  \widehat \alpha_{\B j, l}^{(\rho)} \widehat B_{i,l}^{(\rho)}(\widehat{\cal H}) \right\|^{-1} \right\} \ll P^{ns+\eps} N_{\B j}(P,P),
    \end{align*}
    so it suffices to understand $N_{\B j}(P,P)$. This is an integral lattice problem and can be treated by the usual methods.

    \begin{lem}
        Suppose that $k>0$ and $\theta \in [0,1)$ are parameters and that for some $\U{\ba} \in \T^{Rr}$ one has
        \begin{align*}
            |T_P(\U{\ba})| \gg \Pi^{ms - k \theta}.
        \end{align*}
        Then for any $\B j \in J$ we have
        \begin{align*}
            N_{\B j}(P^\theta, P^{d-(d-1)\theta}) \gg (\Pi^{ \theta})^{(d-1)s-2^{d-1}k}.
        \end{align*}
    \end{lem}

    \begin{proof}
        This follows from the argument leading to \cite[Lemma~3.3]{B:14}. By applying standard results from the geometry of numbers \cite[Lemma~12.6]{Dav:AM} as in the proof of Lemma~3.4 of \cite{B:15FRF2}, it follows that
        \begin{align*}
            N_{\B j}(P^\theta, P^{d-(d-1)\theta}) \gg P^{-(d-1)(1-\theta)ns}N_{\B j}(P, P),
        \end{align*}
        so we find
        \begin{align*}
            |T_P(\U{\ba})|^{2^{d-1}}
            & \ll P^{(2^{d-1}m-d)ns}  P^{ns+\eps} P^{(d-1)(1-\theta)ns} N_{\B j}(P^\theta, P^{d-(d-1)\theta}).
        \end{align*}
        Under the hypotheses of the lemma we have $|T_P(\U{\ba})|^{2^{d-1}} \gg P^{2^{d-1}(mns-nk\theta)}$, and rearranging reproduces the claim.
    \end{proof}

    We may now apply the argument of \cite[Lemma~2]{Sk:97} to each $\U \alpha_{\B j}$ in turn. This is analogous to the procedure of \cite[Lemma~3.4]{B:14}, and as a result we find that, if the exponential sum is large at some value $\U \ba$, then either all components of $\U \ba$ have a good approximation in the $\K$-rational numbers, or else the system of forms $F^{(1)}, \dots, F^{(R)}$ is singular in the sense that the matrix $(B_{i,l}^{(\rho)}(\cal H))_{i,l;\rho}$ has rank less than $R$ for at least $(\Pi^{ \theta})^{(d-1)s - 2^{d-1}k - \eps}$ values of $\cal H \in P^\theta \cal B^{(d-1)s}$. Furthermore, the proof of Lemma~4 in \cite{Sk:97} now carries over unchanged, so if $s - \dim \sing^* \B F > 2^{d-1}k$, then the singular case is excluded. This yields the following tripartite case distinction.
    \begin{lem}\label{l:weyl}
        Let $0 < \theta \le 1$ and $k>0$ be parameters, and suppose that
    \begin{align}\label{eq:sing}
        s - \dim \sing^* \B F > 2^{d-1}k.
    \end{align}
        Then for each $\U{\ba} \in \T^{Rr}$ either
        \begin{enumerate}[(A)]
            \item
                the exponential sum $T_P(\U{\ba})$ is bounded by
                \begin{align*}
                    |T_P(\U{\ba})| \ll \Pi^{ms - k \theta},
                \end{align*}
                or
            \item
                for every $\B j \in J$ one finds $(q_{\B j}, \underline{a}_{\B j}) \in (\OK^+)^{R+1}$ satisfying
                \begin{align*}
                    1 \le |q_{\B j}| \ll P^{(d-1)R\theta}  \quad \hbox{and} \quad
                    \left|\U \alpha_{\B j} q_{\B j} - \U a_{\B j}\right| \ll P^{-d+(d-1)R\theta}.
                \end{align*}
        \end{enumerate}
    \end{lem}
    This result lies at the heart of our analysis in the next section.

\section{Application of the circle method}

    For a suitable parameter $c_1$ write
    \begin{align*}
        \F M_{q, \U a} (P, \theta)= \{\U{\alpha} \in \T ^R: |\alpha^{(\rho)} q - a^{(\rho)}| \le c_1 P^{-d+R(d-1)\theta} \quad (1 \le \rho \le R)\},
    \end{align*}
    and
    \begin{align*}
        \F M_P^*(\theta) = \bigcup_{\substack{q \in \OK^+ \setminus \{ 0\}\\ |q| \le c_1P^{R(d-1)\theta}}} \bigcup_{\substack{\U a \in (\OK^+) ^R\\ |\U a| \le |q|, (q, \U a)=1 }} \F M_{q, \U a}(P, \theta).
    \end{align*}
    We further set $\F M_P(\theta) = (\F M_P^*(\theta))^r$ and $\F m_P(\theta) = \T^{Rr} \setminus \F M_P(\theta)$. Note that the constant $c_1$ can be chosen in such a manner that the major arcs dissection reflects the case distinction of Lemma~\ref{l:weyl}.

    Now suppose that some $\U \ba \in \F M_P(\theta)$ has two distinct approximations, then for some $\B j \in J$ and $1 \le \rho \le R$ there exist two pairs of $\K$-integers $(a_1, q_1)$ and $(a_2, q_2)$ with the property that $\left|q_i\right| \ll P^{R(d-1)\theta}$ and $\left|a_i - \alpha_{\B j}^{(\rho)}q_i\right| \ll P^{-d+(d-1)R\theta}$ for $i \in \{1,2\}$. Hence we have the chain of inequalities
    \begin{align*}
        1 \ll |a_1q_2-a_2q_1| \ll \left|q_2\right| \left|a_1 - \alpha_{\B j}^{(\rho)}q_1\right| + \left|q_1\right| \left|a_2 - \alpha_{\B j}^{(\rho)}q_2\right| \ll P^{-d+2R(d-1)\theta}.
    \end{align*}
    Thus if
    \begin{align}\label{eq:disj}
        2R(d-1)\theta < d,
    \end{align}
    then the major arcs are disjoint.

    By Lemma~5 (iii) of \cite{Sk:97} we have
    \begin{align*}
        \vol \F M_P^*(\theta) \ll \Pi^{-Rd+R(R+1)(d-1)\theta+\eps},
    \end{align*}
    and hence
    \begin{align*}
        \vol \F M_P(\theta) \ll \Pi^{-Rrd+R(R+1)r(d-1)\theta+\eps}.
    \end{align*}
    It is then clear that Lemma~4.1 of \cite{B:14} can be directly transferred to the number field setting.

    \begin{lem}
        Suppose that \eqref{eq:sing} holds and that the parameters $k$ and $\theta$ satisfy
        \begin{align*}
            0 < \theta & < \theta_0 = \frac{d}{(d-1)(R+1)}
        \end{align*}
        and
        \begin{align}\label{eq:k1}
            k >  Rr(R+1)(d-1).
        \end{align}
        Then there exists a $\delta > 0$ such that the minor arcs contribution is bounded by
        \begin{align*}
            \int_{\F m_P(\theta)} |T_P(\U{\ba})| \D \U{\ba} \ll \Pi^{ms-Rrd-\delta}.
        \end{align*}
    \end{lem}

    We now define a second set of major arcs that will be easier to work with. Recall that Lemma~\ref{l:weyl} produces an approximation $\U \alpha_{\B j} =  \U a_{\B j}/ q_{\B j}  + \U \beta_{\B j}$ for each $\B j \in J$ in turn. Taking least common multiples, we find that there is an approximation $\U \ba =  \U {\B a}/ q + \U \bb$ with $|q| \le \prod_{\B j}|q_{\B j}| \ll P^{Rr(d-1)\theta}$ and $|q\U \bb| \ll P^{-d+Rr(d-1)\theta}$. Recall the definition \eqref{eq:denom}, and for $\U \bg \in (\K \cap \T)^{Rr}$ set $q_{\U \bg} = \left|\Nm( \F q(\U \bg))\right|$. In this notation we denote the homogeneous major arcs by
    \begin{align*}
        \F N_{\U \bg} = \{\U{\ba} \in \T^{Rr}: |\alpha_{\B j}^{(\rho)} - \gamma_{\B j}^{(\rho)}| \le c_2 P^{-d+Rr(d-1)n\theta} \quad (1 \le \rho \le R, \B j \in J)\}
    \end{align*}
    and
    \begin{align*}
        \F N(\theta) = \bigcup_{\substack{\U{\bg} \in (\K \cap \T)^{Rr} \\ q_{\U \bg} \le  c_2 P^{Rr(d-1)n\theta}}} \F N_{\U{\gamma}}.
    \end{align*}
    It follows from \cite[Lemma~5 (ii)]{Sk:97} that $c_2$ can be chosen in such a way that $\F M_P(\theta) \subseteq \F N(\theta)$.
    We further let
    \begin{align*}
        S(\U \bg) &= \sum_{\ol x \mmod{\F q(\U{\bg})}} e(\F F(\ol x; \U{\bg})), \\
        v_P(\U{\bb}) &= \int_{P \cal B^{sm}} e(\F F(\ol y; \U{\bb})) \D \ol y,
    \end{align*}
    and set
    \begin{align*}
        \F S(P) & = \sum_{\substack{\U{\bg}\in (\K \cap \T)^{Rr} \\ q_{\U \bg} \le c_2 P^{Rr(d-1)n\theta}}} q_{\U \bg}^{-ms}S(\U \bg), \\
        \F J(P) &=  \int_{|\U{\bb}| \le c_2 P^{-d+Rr(d-1)n\theta}}v_P(\U{\bb})\D \U{\bb}.
    \end{align*}
    In this notation the exponential sum can be approximated by a product of the truncated singular series and integral.

    \begin{lem}\label{l:gf}
        Let $\U \ba \in \T^{Rr}$ be of the shape $\U \ba = \U \bg + \U \bb$ with $\U \bg \in (\K \cap \T)^{Rr}$. Then we have
        \begin{align*}
            \left|T_P(\U \ba) - q_{\U \bg}^{-ms} S(\U \bg) v_P(\U{\bb})\right|  \ll q_{\U \bg} P^{mns-1}\left(1 + P^{d} \sum_{\rho=1}^R \sum_{\B j \in J} |\beta_{\B j}^{(\rho)}|\right).
        \end{align*}
    \end{lem}

    \begin{proof}
        This is \cite[Lemma~5.2]{FM} specified to our situation.
    \end{proof}
    We can now integrate over the major arcs $\F N(\theta)$. Their volume is easily computed using the fact that $\vol \F N_{\U \bg} \ll (P^{-d+(d-1)nRr\theta})^{nRr}$. Thus, using \eqref{eq:nmineq}, we have
    \begin{align*}
        \vol \F N(\theta) &\ll \sum_{q=1}^{c_2 P^{Rr(d-1)n\theta}} \sum_{\substack{\U \bg \in (\K \cap \T)^{Rr} \\ q_{\U \bg} =q }} \vol \F N_{\U \bg} \ll P^{-nRrd + ((n+1)Rr+1)Rr(d-1)n\theta+\eps}.
    \end{align*}
    It follows that
    \begin{align*}
        \int_{\F N(\theta)} |T_P(\U{\ba})| \D \U{\ba} - \F S(P)\F J(P) &\ll \vol \F N(\theta )  \sup_{\U{\ba} = \U \bg + \U \bb\in \F N(\theta)}  |T_P(\U{\ba}) -  q_{\U \bg}^{-ms} S(\U \bg) v_P(\U{\bb})|\\
        & \ll P^{mns-nRrd -1+ ((n+1)Rr+3)Rr(d-1)n\theta+\eps}.
    \end{align*}
    It is clear that this is dominated by $\Pi^{ms-Rrd-\delta}$ for some $\delta>0$ whenever $\theta$ has been chosen small enough.
    Furthermore, a standard rescaling shows that
    \begin{align}\label{eq:intp}
        v_P(\U{\bb}) = \Pi^{ms} v_1( P^{d} \U{\bb}),
    \end{align}
    and therefore
    \begin{align*}
      \F J(P) =  \Pi^{ms-Rrd} \int_{|\U{\bb}| \le c_2 P^{Rr(d-1)n\theta}} v_1(\U{\bb}) \D \U{\bb}.
    \end{align*}
    It thus remains to see that the limits $\F S = \lim_{P \to \infty}\F S(P)$ of the singular series and $\F J = \lim_{P \to \infty}\Pi^{-ms+Rrd}\F J(P)$ of the rescaled singular integral exist.

    \begin{lem}\label{l:ss}
        Let $k$ be as in Lemma~\ref{l:weyl}. For any $\U \bg \in (\T \cap \K)^{Rr}$ we have
        \begin{align*}
            q_{\U \bg}^{-ms} |S(\U \bg)| \ll q_{\U \bg}^{-\frac{k}{R(d-1)}}.
        \end{align*}
    \end{lem}

    \begin{proof}
        Here we follow the treatment of \cite[Lemma~4.1 resp. 7.1]{B:15FRF3}, which is in turn a simplification of \cite[Lemma~8.2]{BHB:14}. From combining Lemma~\ref{l:gf} with \eqref{eq:intp} and observing that $v_1(\U \bb) \asymp 1$, it follows that the relation
        \begin{align}\label{eq:ss}
            q_{\U \bg}^{-ms} |S(\U \bg)| \ll Q^{-mns}|T_Q(\U \bg)| + Q^{-1}q_{\U \bg}
        \end{align}
        holds for any parameter $Q$. We set $Q = q_{\U \bg}^A$ for some suitably large parameter $A$. Take $q \in \F q(\U \bg) \setminus \{0\}$ such that $|q|$ is minimal, then it follows from Minkowski's Theorem that $q_{\U \bg} \gg |q|^n$. Fix $\theta$ such that $|q| = c_1Q^{(d-1)R\theta}$, so that $\U \bg \in \F M_Q(\theta)$. Observe further that by taking $A$ large enough we may assume that \eqref{eq:disj} is satisfied, so the major arcs are disjoint and $\U \bg$ lies just on the edge of the major arcs $\F M_Q(\theta)$. By continuity, the minor arcs bound for $T_Q(\U \bg)$ is still applicable on the boundary of the minor arcs, and we find from Lemma~\ref{l:weyl} (A) that
        \begin{align*}
            Q^{-mns}|T_Q(\U \bg)| \ll Q^{-nk \theta} \ll |q_{\U \bg}|^{-\frac{k}{R(d-1)}}.
        \end{align*}
        The proof is now complete upon inserting this bound into \eqref{eq:ss} and choosing $A$ sufficiently large.
    \end{proof}

    With the help of Lemma~\ref{l:ss} we can show that the singular series converges.
    In fact, by \eqref{eq:nmineq} we have
    \begin{align*}
        \F S = \sum_{\U{\bg} \in (\K \cap \T)^{Rr}} q_{\U \bg}^{-ms}S(\U \bg) 
        & \ll  \sum_{q=1}^\infty q^{-\frac{k}{R(d-1)}} \sum_{\substack{\U{\bg} \in (\K \cap \T)^{Rr}\\ q_{\U \bg} = q}} 1
         \ll \sum_{q=1}^\infty q^{Rr-\frac{k}{R(d-1)} + \eps},
    \end{align*}
    and this sum converges whenever
    \begin{align}\label{eq:K2}
        k > R(d-1)(Rr+1).
    \end{align}

    We now turn to the completion of the singular integral.
    \begin{lem}\label{l:si}
        For any $\U \bb \in \V^{Rr}$ we have
        \begin{align*}
            |v_1(\U \bb)| \ll (1+|\U \bb|)^{-\frac{nk}{R(d-1)}}.
        \end{align*}
    \end{lem}
    \begin{proof}
        This is similar to the previous lemma. Observe that the statement is trivial for $|\U \bb| \le 1$, so we may assume $|\U \bb| > 1$ for the remainder of the argument. By taking $\U{\B a} = \U{\bm 0}$ and $q = 1$, Lemma~\ref{l:gf} together with \eqref{eq:intp} show for any $Q$ that
        \begin{align}\label{eq:v-bd}
            |v_1(\U \bb)| = Q^{-mns} |v_Q(Q^{-d}\U \bb)| \ll Q^{-mns} |T_Q(Q^{-d} \U \bb)| + Q^{-1}|\U \bb|,
        \end{align}
        where we used that $S(\U {\bm 0}) = 1$. We now set $Q = |\U \bb|^A$ for some suitably large parameter $A$ and determine $\theta$ such that $|\U \bb| =c_1 Q^{(d-1)R\theta}$, so that $P^{-d} \U \bb \in \F M_Q(\theta)$ with approximation $\U{\B a} = \U{\bm 0}$ and $q = 1$. Furthermore, by choosing $A$ large enough we can enforce \eqref{eq:disj}, so we may assume the major arcs to be disjoint.
        As in the previous lemma, this implies that the point $Q^{-d}\U \bb$ lies just on the edge of the major arcs in a region where the minor arcs bound of Lemma~\ref{l:weyl} is still valid. This leads to the complementary bound
        \begin{align*}
            Q^{-mns}|T_Q(Q^{-d} \U \bb)| \ll Q^{-nk\theta} \ll |\U \bb|^{- \frac{nk}{R(d-1)}}.
        \end{align*}
        On inserting this into \eqref{eq:v-bd}, we see that
        \begin{align*}
            |v_1(\U \bb)|  \ll |\U \bb|^{-\frac{nk}{(d-1)R}}+ Q^{-1}|\U \bb| = |\U \bb|^{-\frac{nk}{(d-1)R}}+ |\U \bb|^{1-A},
        \end{align*}
        which is satisfactory whenever $A$ has been chosen large enough.

    \end{proof}

    As in the case of the singular series, we can now complete the singular integral. We have
    \begin{align*}
        \int_{|\U \bb| \le X} v_1(\U \bb) \D \U \bb  \ll \int_{|\U \bb| \le X} (1 + |\U \bb|)^{-\frac{nk}{(d-1)R}} \D \U \bb \ll 1+X^{n\left(Rr-\frac{k}{(d-1)R}\right)},
    \end{align*}
    from whence it follows that the limit $X \to \infty$ exists as soon as \eqref{eq:K2} holds. Finally, we take note that \eqref{eq:K2} is strictly implied by \eqref{eq:k1}. This proves Theorem \ref{thm:main}.

\section{The local factors}\label{s:local}
    It is a consequence of the Chinese Remainder Theorem~that we have the product representation
    \begin{align*}
        \F S = \prod_{\F p \subseteq \OK \; \textrm{prime}} \chi_{\F p},
    \end{align*}
    where
    \begin{align*}
        \chi_{\F p} &= \sum_{j=0}^\infty \sum_{\substack{\U{\bg} \in (\K \cap \T)^{Rr} \\ \F q(\U \bg) = \F p^j}} \left|\Nm \F p\right|^{-jms}S(\U \bg).
    \end{align*}
    Furthermore, a straightforward modification of standard arguments as in \cite[Chapter 5]{Dav:AM} shows that this product converges, and furthermore that the factors can be rewritten as
    \begin{align*}
        \chi_{\F p}&= \lim_{j \to \infty} \left|\Nm \F p\right|^{-jms} \sum_{\ol x \mmod{\F p^j}} \sum_{\substack{\U{\bg} \in (\K \cap \T)^{Rr} \\ \F p^j \subseteq \F q(\U \bg)}} e(\F F(\ol x; \U \bg))\\
        & = \lim_{j \to \infty} \left|\Nm \F p\right|^{j(Rr-ms)} \Gamma(\F p^j),
    \end{align*}
    where
    \begin{align*}
        \Gamma(\F p^j) = \card\{\ol x \mmod{\F p^j}: \Phi_{\B j}^{(\rho)}(\ol x) \in \F p^j \quad (1 \le \rho \le R, \B j \in J)\}.
    \end{align*}
    Let $v=v(\F p)$ denote the place associated to the prime ideal $\F p$, then we will equivalently write $\chi_{\F p} = \chi_{v(\F p)}$.
    For $v \in \Omega_0(\K)$ let $\gamma_\K^{(v)}(R,m,d)$ denote the smallest integer $\gamma$ such that any system of $R$ forms of degree $d$ over $\K$ contains an $m$-dimensional linear subspace in $\K_v$, and write $\gamma_\K^{(0)}(R,m,d) = \max_{v \in \Omega_0(\K)} \gamma_\K^{(v)}(R,m,d)$.  Then we have a lower bound for bound for $\Gamma(\F p^j)$ which suffices to show that the local factor $\chi_{\F p}$ is positive.

    \begin{lem}\label{l:Gam2}
        We have
        \begin{align*}
            \Gamma(\F p^j) \gg \left|\Nm \F p\right|^{j(ms-\gamma_\K^{(\F p)}(R,m,d))},
        \end{align*}
        and thus $\chi_{\F p}\gg 1$ whenever
        \begin{align*}
            k > (d-1)R\gamma_\K^{(\F p)}(R,m,d).
        \end{align*}
        Here $k$ is the parameter of Lemma~\ref{l:weyl}.
    \end{lem}
    \begin{proof}
        The first statement is an adaptation of Schmidt \cite[Lemma~2]{Sch:82-4} (see also \cite[Lemma~4.4]{B:PhD}). The proof uses a combinatorial argument involving cyclic subgroups of the additive group $(\OK / \F p^j)^{ms}$, which carries over to number fields without difficulties. The second statement is easily obtained by adapting the arguments of \cite[\S 7]{B:14}.
    \end{proof}
    The quantity $\gamma_\K^{(\F p)}(R,m,d)$ can be bounded by results from the literature. For instance, Wooley \cite[Theorem~2.4]{W:98loc} shows that
    \begin{align*}
        \gamma_\K^{(0)}(R,m,d) \le (R^2 d^2 + mR)^{2^{d-2}} d^{2^{d-1}}
    \end{align*}
    for all algebraic number fields $\K$.

    We also record an alternative bound of a more geometric flavour. Define the singular locus of the expanded system \eqref{eq:Phi} as
    \begin{align*}
        \sing_m \B F = \sing \bm\Phi \subset \A_\K^{ms}.
    \end{align*}
    In this notation \cite[Theorem~5.1]{B:15p} shows that $\Gamma(\F p^j) \gg \left|\Nm \F p\right|^{j(ms-Rr)}$, and hence $\chi_{\F p} \gg 1$, as soon as
        \begin{align*}
            ms - \dim \sing_m \B F \ge \gamma_\K^{(\F p)}(R,m,d).
        \end{align*}
    The proof rests only on Hensel's Lemma and a geometric argument, both of which carry over to the number field setting unchanged. 
    
    It remains to consider the singular integral
    \begin{align*}
        \chi_{\infty} = \int_{\V^{Rr}} v_1(\U \bb) \D \U \bb.
    \end{align*}
    As in \cite[\S 6]{Sk:94}, we observe that $v_1(\U \bb)$ factorises as a product over the infinite places of $\K$. Recall the notation $x^{(l)}$ for the projection of $x$ onto $\K_l$, then we have
    \begin{align*}
        v_1(\U \bb) = \prod_{l=1}^{n_1+n_2} v_1^{(l)}(\U \bb^{(l)}),
    \end{align*}
    where the factors are given by
    \begin{align*}
        v_1^{(l)}(\U \bb^{(l)}) = \int_{[-1,1]^{ms}} e( \F F^{(l)}(\ol x^{(l)}; \U \bb^{(l)})) \D \ol x^{(l)}
    \end{align*}
    in the case $1 \le l \le n_1$ when $\K_l$ is real, and
    \begin{align*}
        v_1^{(l)}(\U \bb^{(l)}) = \int_{[-1,1]^{2ms}} e( 2 \Re \F F^{(l)}(\ol x^{(l)}; \U \bb^{(l)})) \D \Re \ol x^{(l)} \D \Im \ol x^{(l)}
    \end{align*}
    at the complex places $n_1 +1 \le l \le n_1 + n_2$. Correspondingly, we find
    \begin{align*}
        \chi_{\infty} = \int_{\V^{Rr}}\prod_{l=1}^{n_1+n_2} v_1^{(l)}(\U \bb^{(l)}) \D \U \bb = \prod_{l=1}^{n_1+n_2} \int_{\K_l^{Rr}}v_1^{(l)}(\U \bb^{(l)}) \D \U \bb^{(l)} = \prod_{v \in \Omega_\infty(\K)} \chi_{v}.
    \end{align*}
    It remains to investigate under what conditions these factors are positive. For $v \in \Omega_\infty(\K)$ we define
    \begin{align*}
        \F M_v = \{\ol x \in \A_{\K_v}^{ms}: \eta_v(\Phi_{\B j}^{(\rho)})(\ol x) =0 \quad (1 \le \rho \le R, \B j \in J) \}.
    \end{align*}
    Then the methods of Schmidt \cite{Sch:82-4, Sch:82quad} apply.

    \begin{lem}\label{l:si+}
        Suppose that \eqref{eq:K2} is satisfied. We have $\chi_v \gg 1$ whenever $\dim \F M_v \ge ms-Rr$. In particular, this is the case whenever the manifold in question contains a non-singular point. It is always satisfied when $d$ is odd or $\K_v = \C$.
    \end{lem}

    \begin{pf}
        In the case $\K_v = \R$, the first statement is due to Schmidt \cite[Lemma~2 and \S 11]{Sch:82quad} (see also \cite[Chapter 4.5]{B:PhD}), but the proof can be adapted without difficulties to the complex case as well. In order to simplify notation we will suppress the dependence on the embedding $v$. For $L>0$ set
        \begin{align*}
            \hat w_L(x) &=  \max\{0, L(1-L|x|)\} \quad (x \in \R), \\ w_L(z) &= \hat w_L(\Re z) \hat w_L(\Im z) \quad (z \in \C),
        \end{align*}
        and define
        \begin{align*}
            \F J_L = \int_{[-1,1]^{2ms}} \prod_{\rho = 1}^R \prod_{\B j \in J} w_L(\Phi_{\B j}^{(\rho)}(\ol x)) \D  \Re \ol x  \D  \Im \ol x.
        \end{align*}
        The proof of \cite[Lemma~2]{Sch:82quad} (see also \cite[Lemma 4.7]{B:PhD}) can now be adapted in a straightforward manner by interpreting $\C$ as a two-dimensional $\R$-vector space. This shows that under the hypothesis of the statement we have $\F J_L \gg 1$ uniformly in $L$.

        In order to show that $\F J_L \to \F J$ as $L$ tends to infinity, we follow the argument of \cite[\S 11]{Sch:82quad} (see also \cite[Lemma 4.6]{B:PhD}) by considering real and imaginary parts separately. Since
        \begin{align*}
            \hat w_L(x) = \int_{\R}e(\beta x)\left(\frac{\sin(\pi \beta/L)}{\pi \beta/L}\right)^2 \D \beta
        \end{align*}
        and furthermore $\hat w_L(x)=\hat w_L(-x)$, it is easy to show that
        \begin{align*}
            w_L(z) &= \int_{\C}e(\Tr z \beta )\prod_{i=1,2}\left(\frac{\sin(\pi \beta_i/L)}{\pi \beta_i/L}\right)^2\D \beta,
        \end{align*}
        where we set $\beta = \beta_1 + \mathrm i \beta_2$. The argument of \cite[\S 11]{Sch:82quad} can now be adapted easily to show that $\F J - \F J_L \ll L^{-1}$, provided that \eqref{eq:K2} is satisfied. This completes the proof of the first statement of the lemma.

        It thus remains only to comment on the fact that the inequality $\dim \F M_v \ge ms-Rr$ is really satisfied under the stated conditions. If the manifold $\F M_v$ contains a non-singular point, the statement follows from the Implicit Function Theorem, and it is a consequence of basic algebraic geometry if $\K_v=\C$ is algebraically closed (\cite[Chapter I.6, Corollary 1.7]{Sh:BAG1}). Finally, when $\K_v = \R$ and $d$ is odd, the same conclusion has been established by Schmidt \cite[\S 2]{Sch:82-4}.
    \end{pf}

     Theorem~\ref{thm:local} is now immediate upon combining all estimates hitherto obtained. Furthermore, we have the stronger statement that
    \begin{align*}
        N_m(P) = \Pi^{ms-Rrd} \prod_{v \in \Omega(\K)} \chi_v + O(\Pi^{ms-Rrd-\delta}),
    \end{align*}
    where the product over all places of $\K$ converges absolutely, provided the hypotheses of Theorem~\ref{thm:main} are true, and the main term is positive if additionally either $d$ is odd or $\K$ is totally imaginary, and furthermore either of the two conditions 
    \begin{align*}
        ms-\dim \sing_m \B F &\ge d^{2^{d-1}}(R^2d^2+Rm)^{2^{d-2}}
    \end{align*}
    and 
    \begin{align*}
        s - \dim \sing^* \B F > 2^{d-1}(d-1)R d^{2^{d-1}}(R^2d^2+Rm)^{2^{d-2}}
    \end{align*}           
    is satisfied.

\bibliographystyle{amsplain}
\bibliography{fullrefs}

\end{document}